\documentclass[twoside,12pt,a4paper]{amsart} 
\usepackage[utf8]{inputenc}
\usepackage[english]{babel}

\usepackage{amsmath}
\usepackage{amsfonts}
\usepackage{amssymb}
\usepackage{color}

\usepackage{amsthm}
\newtheorem{theorem}{Theorem}[section]
\newtheorem{lemma}[theorem]{Lemma}
\newtheorem{corollary}[theorem]{Corollary}
\usepackage{graphics}
\usepackage{epsfig}

\usepackage{url}
\usepackage{hyperref}

\usepackage[a4paper,top=2.5cm,bottom=2.8cm,
left=3.2cm,right=3.2cm]{geometry}
\markleft{\hfill \textsc{Relationships Between Six Circumcircles} \hfill}
\begin{document}
% ===================
Sangaku Journal of Mathematics (SJM) \copyright SJM \\
ISSN 2534-9562 \\
Volume 3 (2019), pp. 98-105  \\
Received 19 September 2019. Published on-line 28 October 2019 \\ 
web: \url{http://www.sangaku-journal.eu/} \\
\copyright The Author(s) This article is published 
with open access\footnote{This article is distributed under the terms of the Creative Commons Attribution License which permits any use, distribution, and reproduction in any medium, provided the original author(s) and the source are credited.}. \\
% ===========================   
\bigskip
\bigskip

\begin{center}
{\Large \textbf{Relationships Between Six Circumcircles}} \\
\medskip
\bigskip
\textsc{Stanley Rabinowitz} \\
545 Elm St Unit 1,  Milford, New Hampshire 03055, USA \\
e-mail: \href{mailto:stan.rabinowitz@comcast.net}{stan.rabinowitz@comcast.net} \\
web: \url{http://www.StanleyRabinowitz.com/} \\
\end{center}
\bigskip

% ==============================
\textbf{Abstract.} If $P$ is a point inside $\triangle ABC$, then the cevians
through $P$ divide $\triangle ABC$ into six small triangles.
We give theorems about the relationships between the radii of the
circumcircles of these triangles. We also state some results about the
relationships between the circumcenters of these triangles.



\bigskip
\bigskip
% ================================
% 1 Introduction 
% ================================
\section{Introduction}

\newcommand{\degrees}{^\circ}

Let $P$ be any point inside a triangle $ABC$. The cevians
through $P$ divide $\triangle ABC$ into six smaller triangles, labeled $T_1$ through $T_6$
as shown in Figure \ref{fig:sixTriangles}.

\begin{figure}[h!t]
\centering
\includegraphics[width=0.45\linewidth]{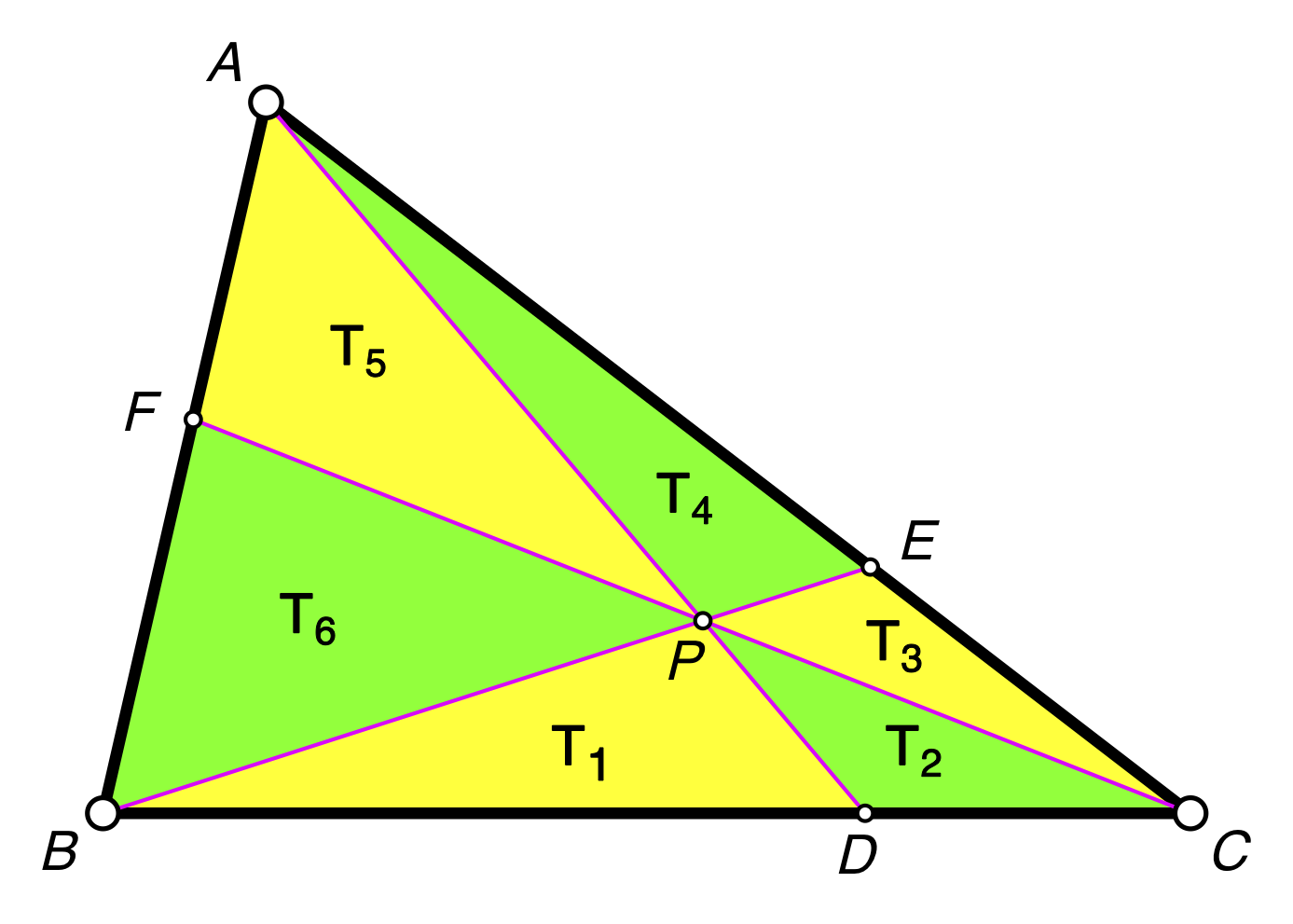}
\caption{numbering of the six triangles}
\label{fig:sixTriangles}
\end{figure}

The relationships between the radii of the circles inscribed in these triangles was investigated in \cite{Rabinowitz-in}.
The relationships between the radii of certain excircles associated with these triangles was investigated in \cite{Rabinowitz-ex}.
In this paper, we will investigate the relationships between the radii of the circles circumscribed about these triangles.

%Notation: If $X$ and $Y$ are points, then we use the notation $XY$ to denote either the line segment joining $X$ and $Y$ or the length of that line segment, depending on the context.

\medskip
We will make use of The Extended Law of Sines which states that if $a$, $b$, and $c$ are the lengths of the sides of a triangle opposite angles $A$, $B$, and $C$, then
$${a\over\sin A}={b\over\sin B}={c\over\sin C}=2R$$
where $R$ is the circumradius of $\triangle ABC$.

\section{Radii}

We begin with some relationships between the radii of the six circumcircles.

\begin{theorem}
\label{thm:general}
Let $P$ be any point inside $\triangle ABC$.
The cevians through $P$ divide $\triangle ABC$ into six smaller triangles, labeled $T_1$ through $T_6$
as shown in Figure \ref{fig:sixTriangles}.
Let $R_i$ be the circumradius of $T_i$.
Then $R_1R_3R_5=R_2R_4R_6$.
\end{theorem}

\begin{proof}
By The Extended Law of Sines in $\triangle PBD$, we have
$$R_1={BD\over 2\sin\angle BPD},$$
with similar expressions for the other $R_i$.
Thus,
$$R_1R_3R_5={BD\over 2\sin\angle BPD}\cdot{CE\over 2\sin\angle CPE}\cdot{AF\over 2\sin\angle APF}\phantom{.}$$
and
$$R_2R_4R_6={DC\over 2\sin\angle DPC}\cdot{EA\over 2\sin\angle EPA}\cdot{FB\over 2\sin\angle FPB}.$$

But $BD\cdot CE\cdot AF=DC\cdot EA\cdot FB$ by Ceva's Theorem.
Also, angles $\angle BPD$ and $\angle EPA$ are vertical angles, so they are congruent and their sines are equal.
Similarly, $\sin\angle CPE=\sin\angle FPB$ and $\sin\angle APF=\sin\angle DPC$.
Therefore, we conclude that $R_1R_3R_5=R_2R_4R_6$.
\end{proof}

\begin{corollary}
\label{cor:9point}
Let $P$ be any point inside $\triangle ABC$.
The cevians through $P$ divide $\triangle ABC$ into six smaller triangles, labeled $T_1$ through $T_6$
as shown in Figure \ref{fig:sixTriangles}.
Let $r_i$ be the radius of the nine-point circle of $T_i$ (the circle through the midpoints of the sides).
Then $r_1r_3r_5=r_2r_4r_6$.
\end{corollary}

\begin{proof}
This follows immediately from the fact that $R_i=2r_i$.
\end{proof}

We have some additional results for specific locations of point $P$.

\begin{theorem}
\label{thm:circumcenter}
Let $O$ be the circumcenter of $\triangle ABC$ and assume that $O$ lies inside $\triangle ABC$.
The cevians through $O$ divide $\triangle ABC$ into six smaller triangles, labeled $T_1$ through $T_6$
as shown in Figure \ref{fig:sixTriangles}.
Let $R_i$ be the circumradius of $T_i$.
Then $R_1=R_2$, $R_3=R_4$, and $R_5=R_6$.
\end{theorem}

\begin{figure}[h!t]
\centering
\includegraphics[width=0.45\linewidth]{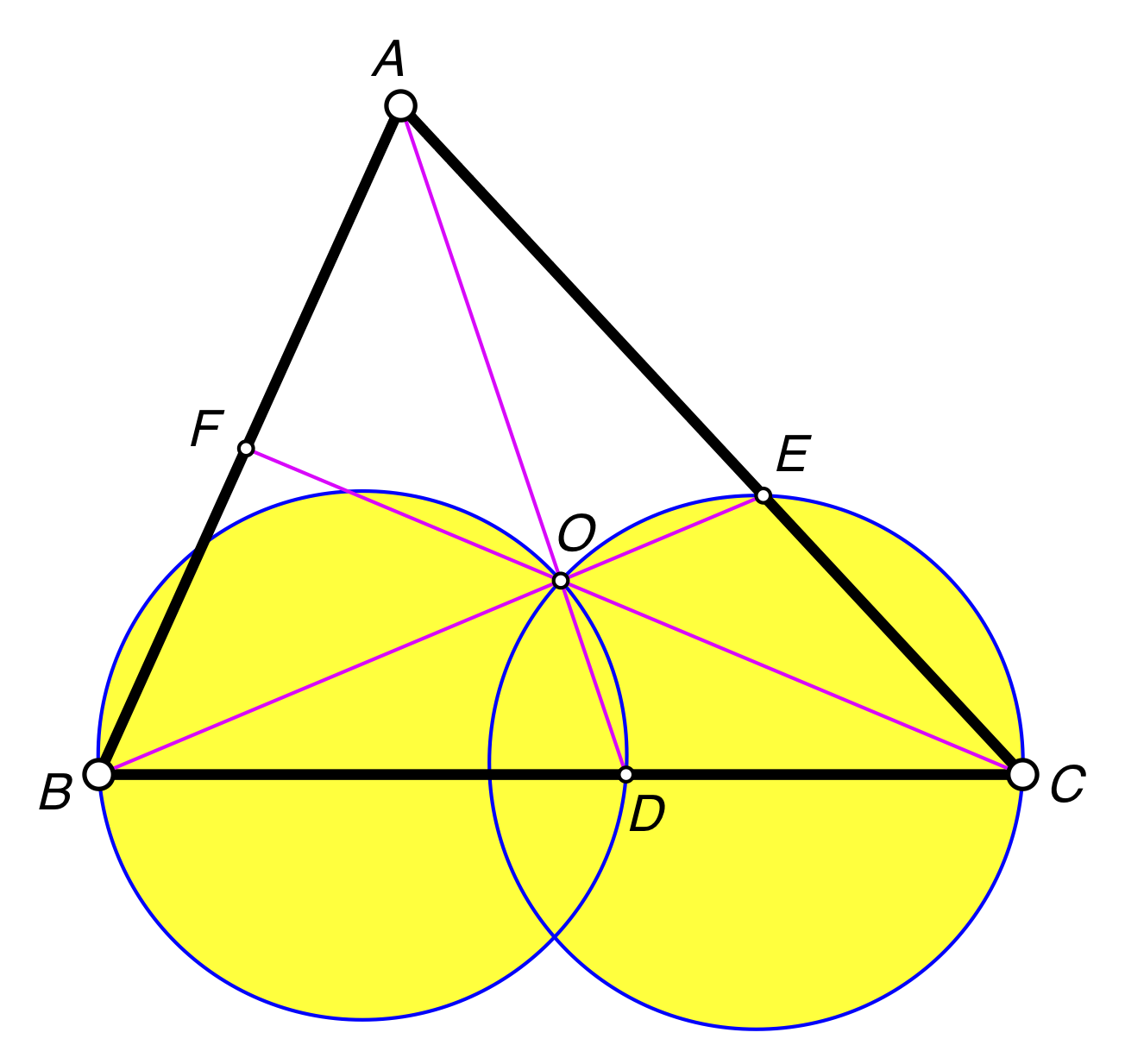}
\caption{Circumcenter: $R_1=R_2$}
\label{fig:circumcenter}
\end{figure}

\begin{proof}
By symmetry, it suffices to show that $R_1=R_2$ (Figure \ref{fig:circumcenter}).
Since $O$ is the circumcenter of $\triangle ABC$, $OB=OC$.
Angles $\angle ODB$ and $\angle ODC$ are supplementary, so their sines are equal.
Thus, by The Extended Law of Sines, we have
$$R_1={OB\over 2\sin\angle ODB}={OC\over 2\sin\angle ODC}=R_2$$
as required.
\end{proof}

\begin{theorem}
\label{thm:Nagel}
Let $N$ be the Nagel Point of $\triangle ABC$.
The cevians through $N$ divide $\triangle ABC$ into six smaller triangles, labeled $T_1$ through $T_6$
as shown in Figure \ref{fig:sixTriangles}.
Let $R_i$ be the circumradius of $T_i$.
Then $R_1=R_4$, $R_2=R_5$, and $R_3=R_6$.
\end{theorem}

\begin{figure}[h!t]
\centering
\includegraphics[width=0.5\linewidth]{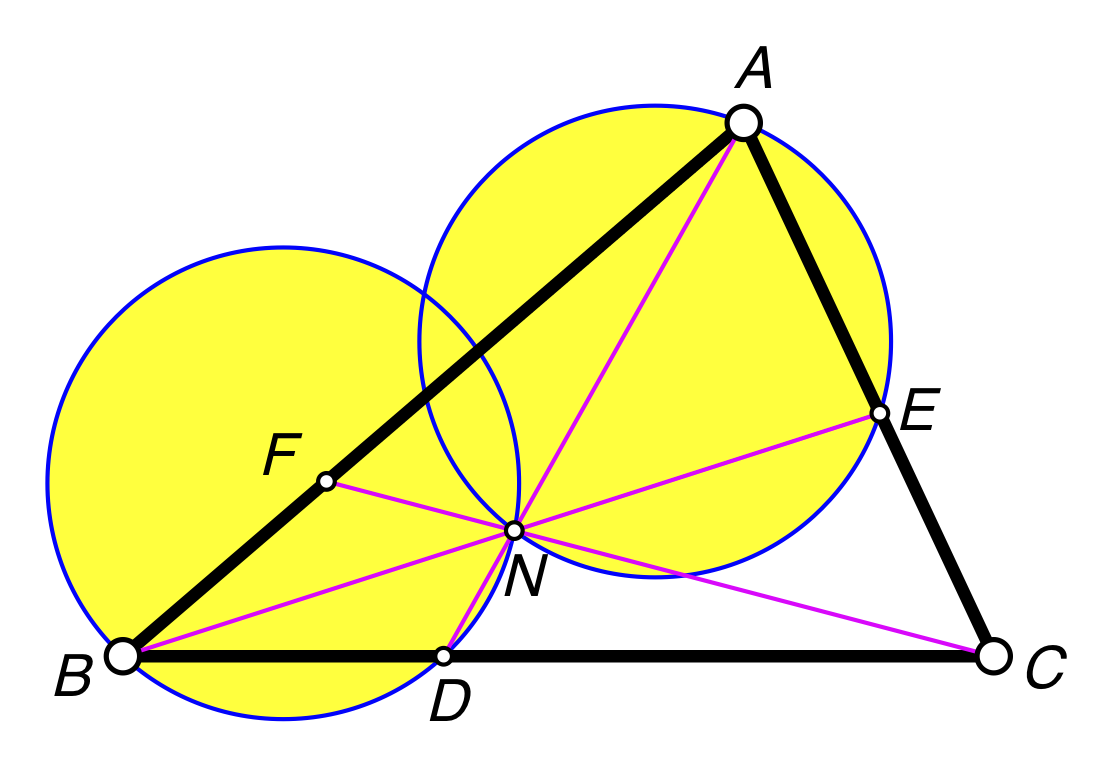}
\caption{Nagel Point: $R_1=R_4$}
\label{fig:Nagel}
\end{figure}

Note: The Nagel Point of a triangle is the point of concurrence of $AD$, $BE$, and $CF$,
where $D$, $E$, and $F$ are the points where the excircles of $\triangle ABC$
touch the sides $BC$, $CA$, and $AB$, respectively \cite[p.~160]{Altshiller-Court}.
The Nagel point is usually denoted $Na$ or $N_a$,
but here we will refer to it as $N$, for simplicity.

\begin{proof}
First note that by symmetry, it suffices to show that $R_1=R_4$ (Figure \ref{fig:Nagel}).
If $BC=a$, $CA=b$, $AB=c$, and $s=(a+b+c)/2$, then it is known that
$BD=AE=s-c$ \cite[p.~88]{Altshiller-Court}.
Thus, by The Extended Law of Sines and the fact that $\angle BND=\angle ENA$, we have
$$R_1={BD\over 2\sin\angle BND}={AE\over 2\sin\angle ENA}=R_4$$
as required.
\end{proof}

\begin{theorem}
\label{thm:orthocenter}
Let $H$ be the orthocenter of $\triangle ABC$.
The cevians through $H$ divide $\triangle ABC$ into six smaller triangles, labeled $T_1$ through $T_6$
as shown in Figure \ref{fig:sixTriangles}.
Let $C_i$ be the circumcircle of $T_i$.
Let $R_i$ be the radius of $C_i$.
Then $R_1=R_6$, $R_2=R_3$, and $R_4=R_5$.
\end{theorem}

\begin{figure}[h!t]
\centering
\includegraphics[width=0.5\linewidth]{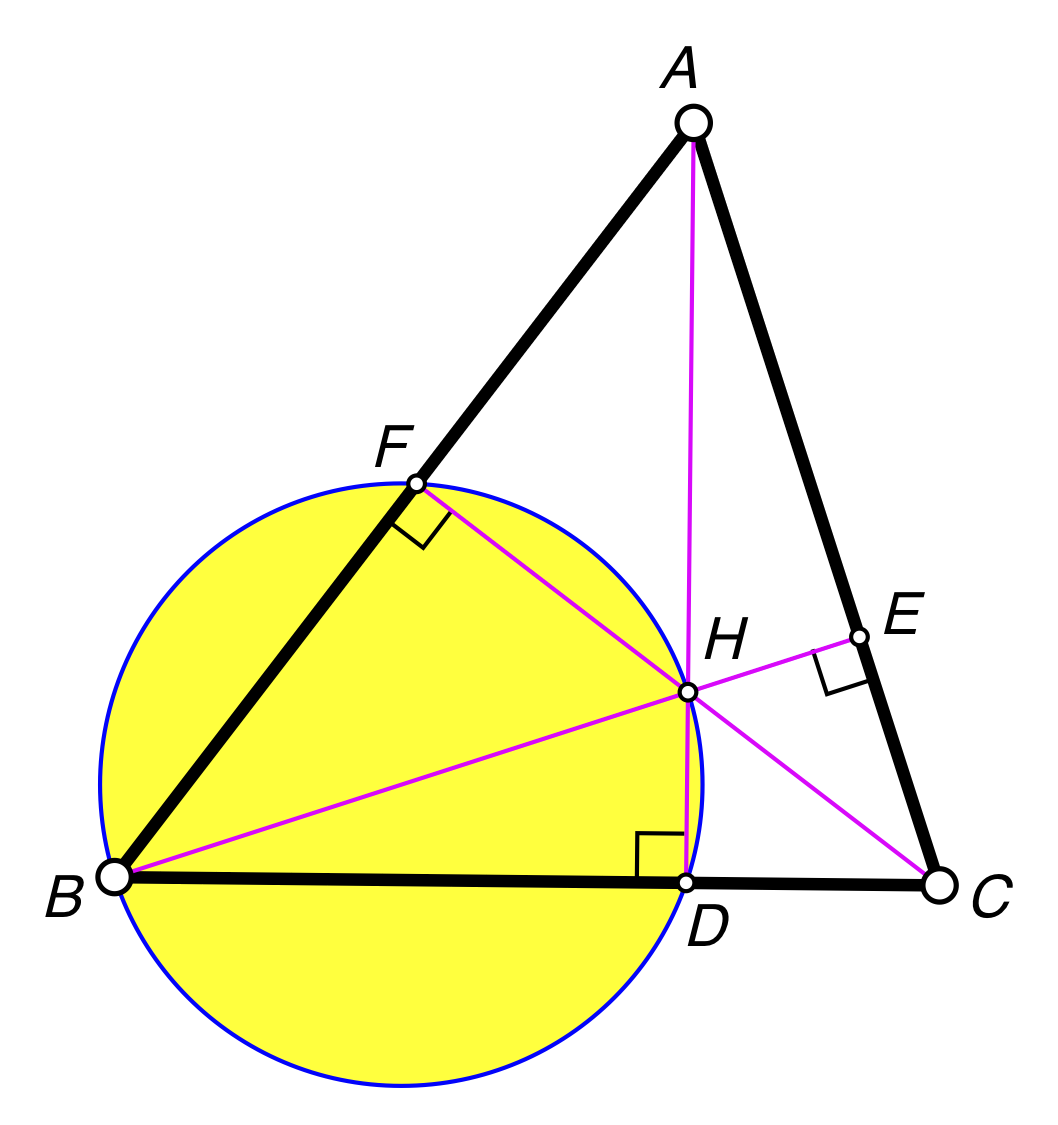}
\caption{Orthocenter: $R_1=R_6$}
\label{fig:orthocenter}
\end{figure}

\begin{proof}
By symmetry, it suffices to show that $R_1=R_4$ (Figure \ref{fig:orthocenter}),
i.e., that $C_1$ and $C_6$ coincide.
Since $\angle BDH+\angle HFB=180\degrees$, quadrilateral $BDHF$ is cyclic.
Thus, the circle through points $B$, $D$, and $H$ is the same as the circle through points $B$, $F$, and $H$.
\end{proof}

%\newpage
\section{Circumcenters}

Now we collect together some interesting results concerning the centers of the six circumcircles.
Most, but not all, of these results are scattered about in the literature.
We will use the notation [XYZ] to denote the area of $\triangle XYZ$.

\begin{theorem}
\label{thm:twoTriangles}
Let $P$ be any point inside $\triangle ABC$.
The cevians through $P$ divide $\triangle ABC$ into six smaller triangles, labeled $T_1$ through $T_6$
as shown in Figure \ref{fig:sixTriangles}.
Let $O_i$ be the circumcenter of $T_i$.
Then $[O_1O_3O_5]=[O_2O_4O_6]$
(Figure \ref{fig:twoTriangles}).
\end{theorem}

\begin{figure}[h!t]
\centering
\includegraphics[width=0.5\linewidth]{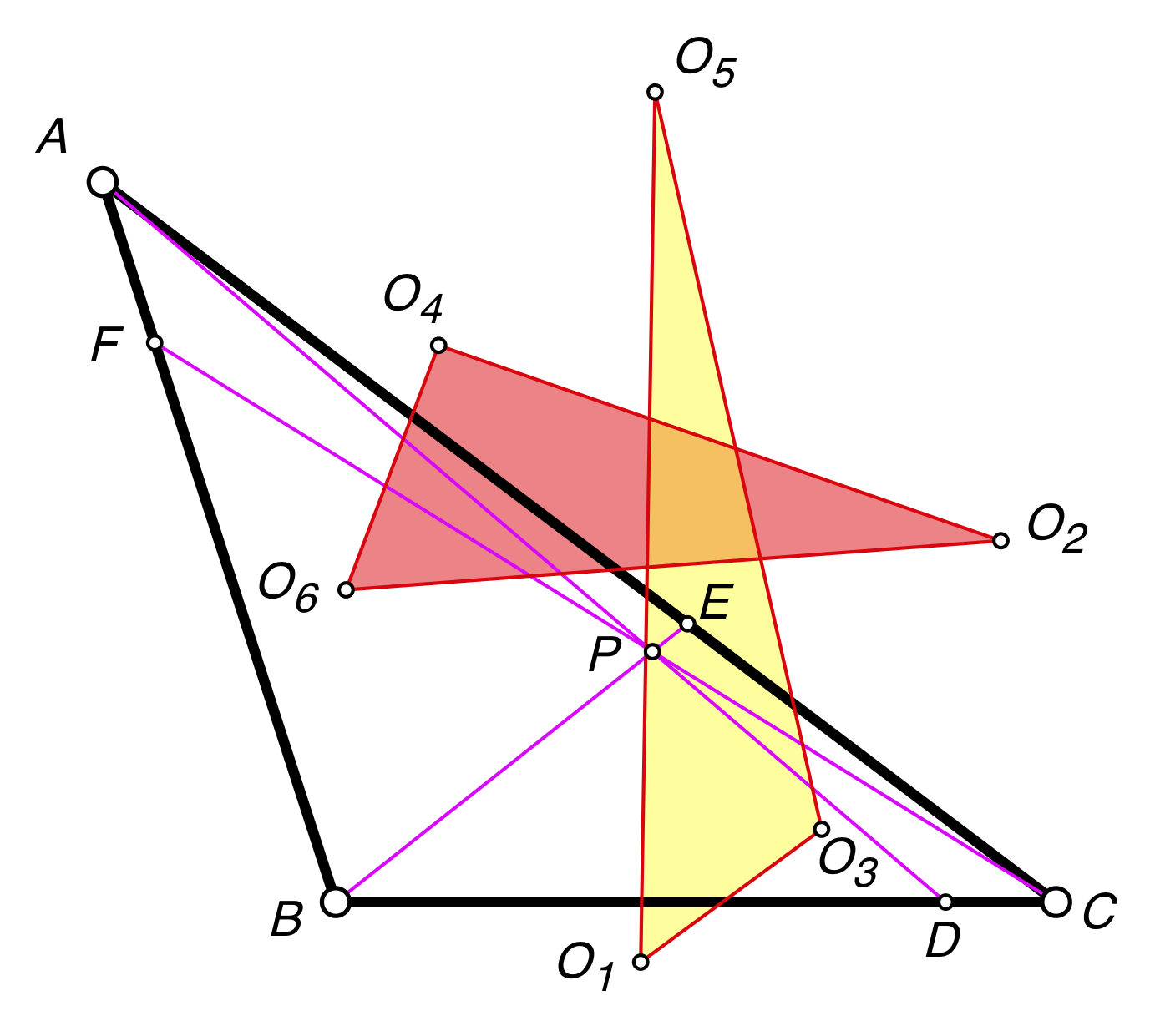}
\caption{Two triangles have same area}
\label{fig:twoTriangles}
\end{figure}

The following proof is due to Dubrovsky \cite{Dubrovsky}.

\begin{proof}
Since $O_1$ is the circumcenter of $\triangle BPD$, it must lie
on the perpendicular bisector of $BP$. The same remark holds true for $O_6$.
Therefore, $O_1O_6\perp BP$. In the same way, $O_6O_5\perp PF$, $O_5O_4\perp AP$,
$O_4O_3\perp PE$, $O_3O_2\perp CP$, and $O_2O_1\perp PD$.
Hence $O_1O_6\parallel O_3O_4$, $O_6O_5\parallel O_2O_3$, and $O_5O_4\parallel O_1O_2$.
Therefore, hexagon $O_1O_2O_3O_4O_5O_6$ has its opposite sides parallel.
But it is known \cite{Kantrowich} that if $ABCDEF$ is a hexagon with its opposite sides parallel,
then $[ACE]=[BDF]$. Thus $[O_1O_3O_5]=[O_2O_4O_6]$.
\end{proof}

\begin{theorem}
\label{thm:equalChords}
Let $M$ be the centroid of $\triangle ABC$.
The medians through $M$ divide $\triangle ABC$ into six smaller triangles, labeled $T_1$ through $T_6$
as shown in Figure \ref{fig:sixTriangles}.
Let $O_i$ be the circumcenter of $T_i$.
Then $O_1O_4=O_2O_5=O_3O_6$.
(Figure \ref{fig:equalChords}).
\end{theorem}

\begin{figure}[h!t]
\centering
\includegraphics[width=0.5\linewidth]{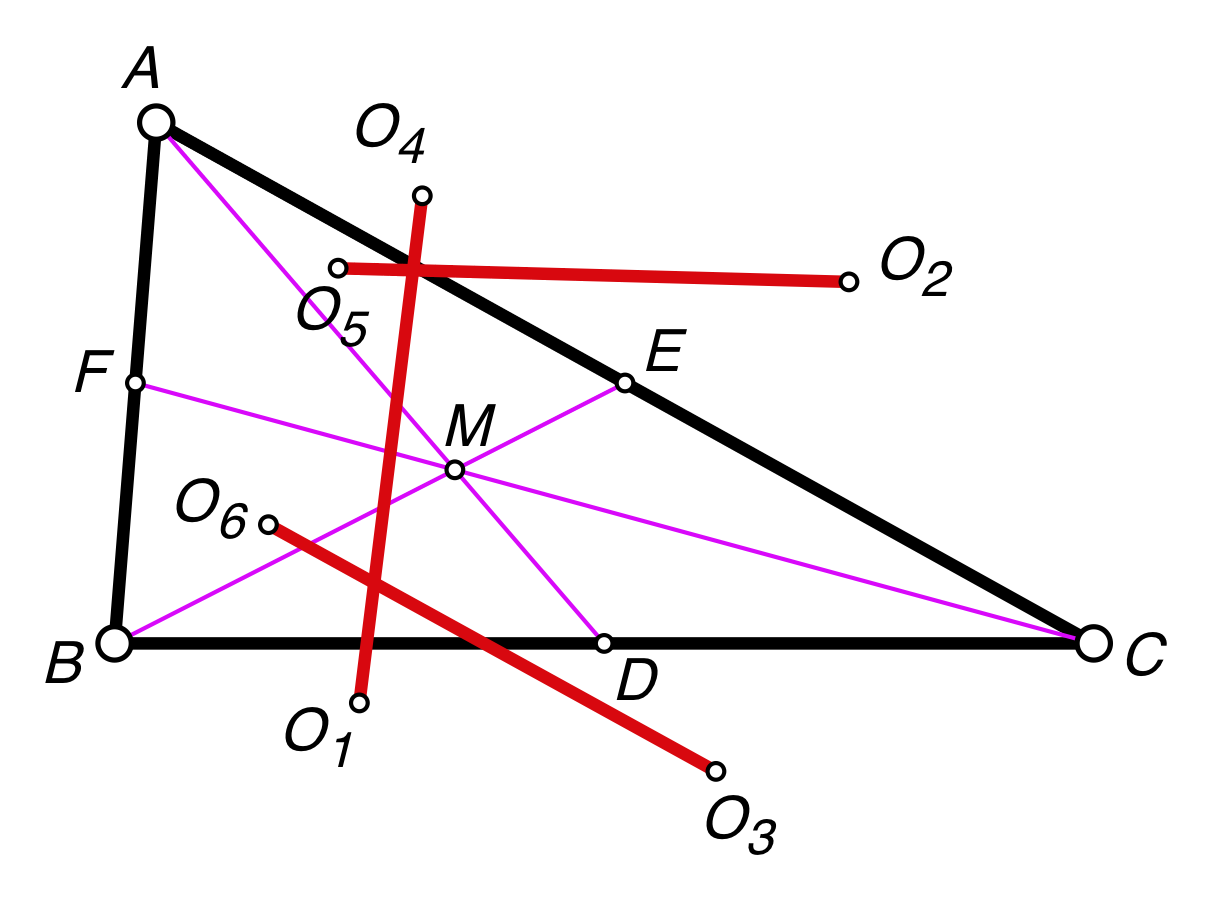}
\caption{Red segments are congruent}
\label{fig:equalChords}
\end{figure}

\begin{proof}
This follows from Proposition 4 of \cite{Myakishev}.
\end{proof}

The following result comes from \cite{Myakishev}.

\begin{theorem}
\label{thm:circle-M}
Let $P$ be any point inside $\triangle ABC$.
The cevians through $P$ divide $\triangle ABC$ into six smaller triangles, labeled $T_1$ through $T_6$
as shown in Figure \ref{fig:sixTriangles}.
Let $O_i$ be the circumcenter of $T_i$.
Then the points $O_i$ lie on a circle if and only if either $P$ is the centroid of $\triangle ABC$
(Figure \ref{fig:circle-M}) or $P$ is the orthocenter of $\triangle ABC$ (in which case $O_6=O_1$, $O_2=O_3$, and $O_4=O_5$).
\end{theorem}

\begin{figure}[h!t]
\centering
\includegraphics[width=0.5\linewidth]{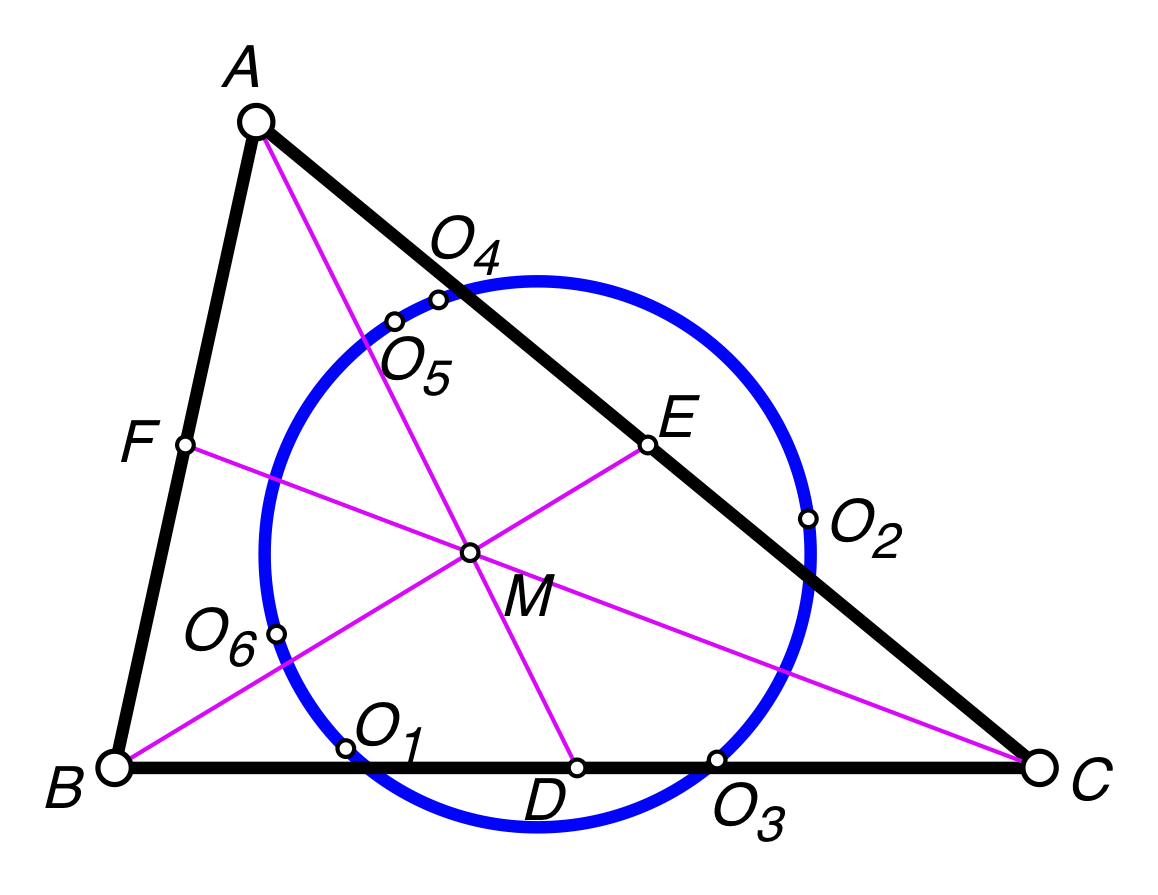}
\caption{$O_i$ lie on a circle when $P=M$}
\label{fig:circle-M}
\end{figure}

The ``if'' portion of this theorem is the well-known Van Lamoen's Theorem, \cite{Lamoen}.

\begin{theorem}
\label{thm:conic}
Let $P$ be any point inside $\triangle ABC$.
The cevians through $P$ divide $\triangle ABC$ into six smaller triangles, labeled $T_1$ through $T_6$
as shown in Figure \ref{fig:sixTriangles}.
Let $O_i$ be the circumcenter of $T_i$.
Then the points $O_i$ lie on a conic (Figure \ref{fig:conic}).
\end{theorem}

\begin{figure}[h!t]
\centering
\includegraphics[width=0.48\linewidth]{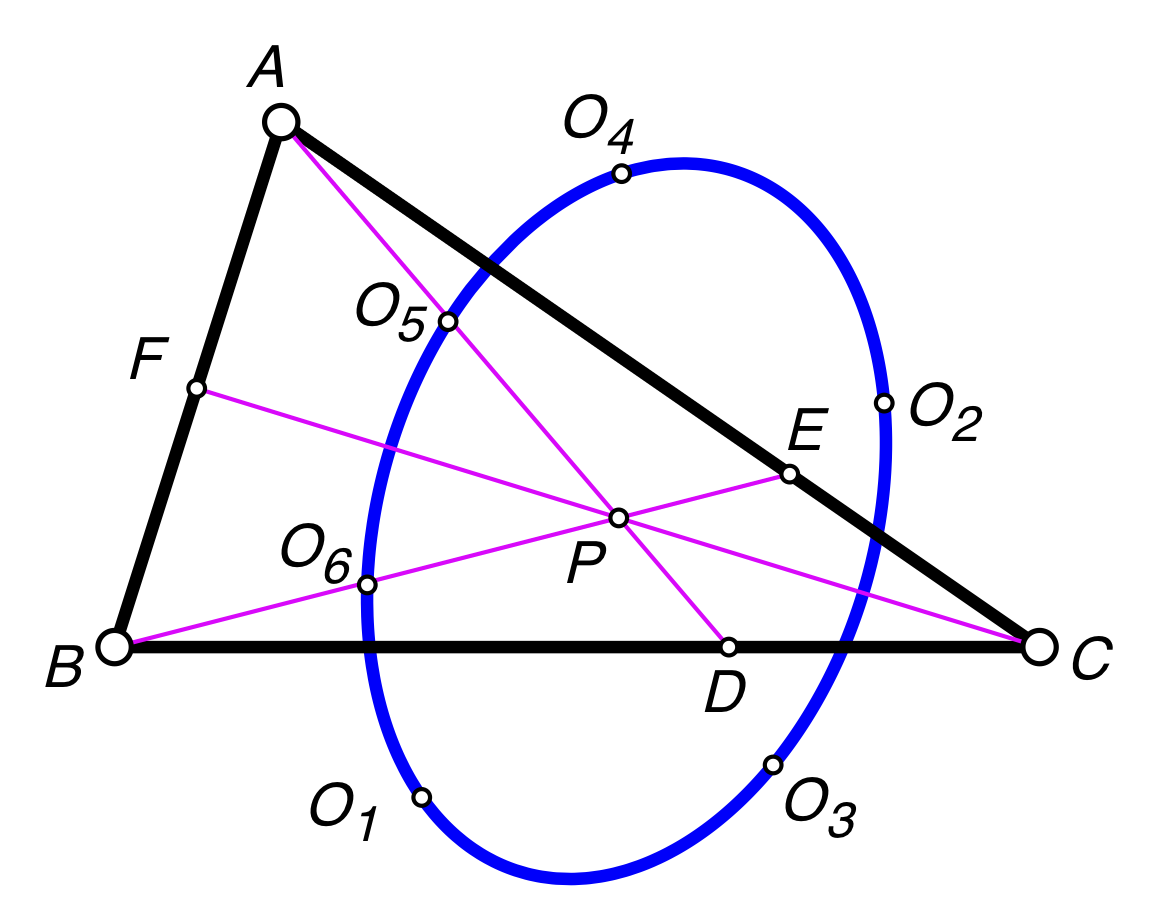}
\caption{$O_i$ lie on a conic}
\label{fig:conic}
\end{figure}

The following proof comes from \cite{Myakishev}.

\begin{proof}
Since $O_1O_6\parallel O_3O_4$, $O_6O_5\parallel O_2O_3$, and $O_5O_4\parallel O_1O_2$,
the result follows from the converse of Pascal's Theorem.
\end{proof}

Before stating our next result, we will need the following lemma
which comes from~\cite{Pavlakis}.

\begin{lemma}
\label{lemma:symmedian}
Let P be any point on the median $AD$ of $\triangle ABC$.
Let $BE$ and $CF$ be the cevians through $P$.
Suppose the circumcircles of triangles $BPF$ and $CEP$ meet at points $P$ and $Q$.
Then $\angle BPQ=\angle CPD$ (Figure \ref{fig:lemmaFigure}).
\end{lemma}

\begin{figure}[h!t]
\centering
\includegraphics[width=0.5\linewidth]{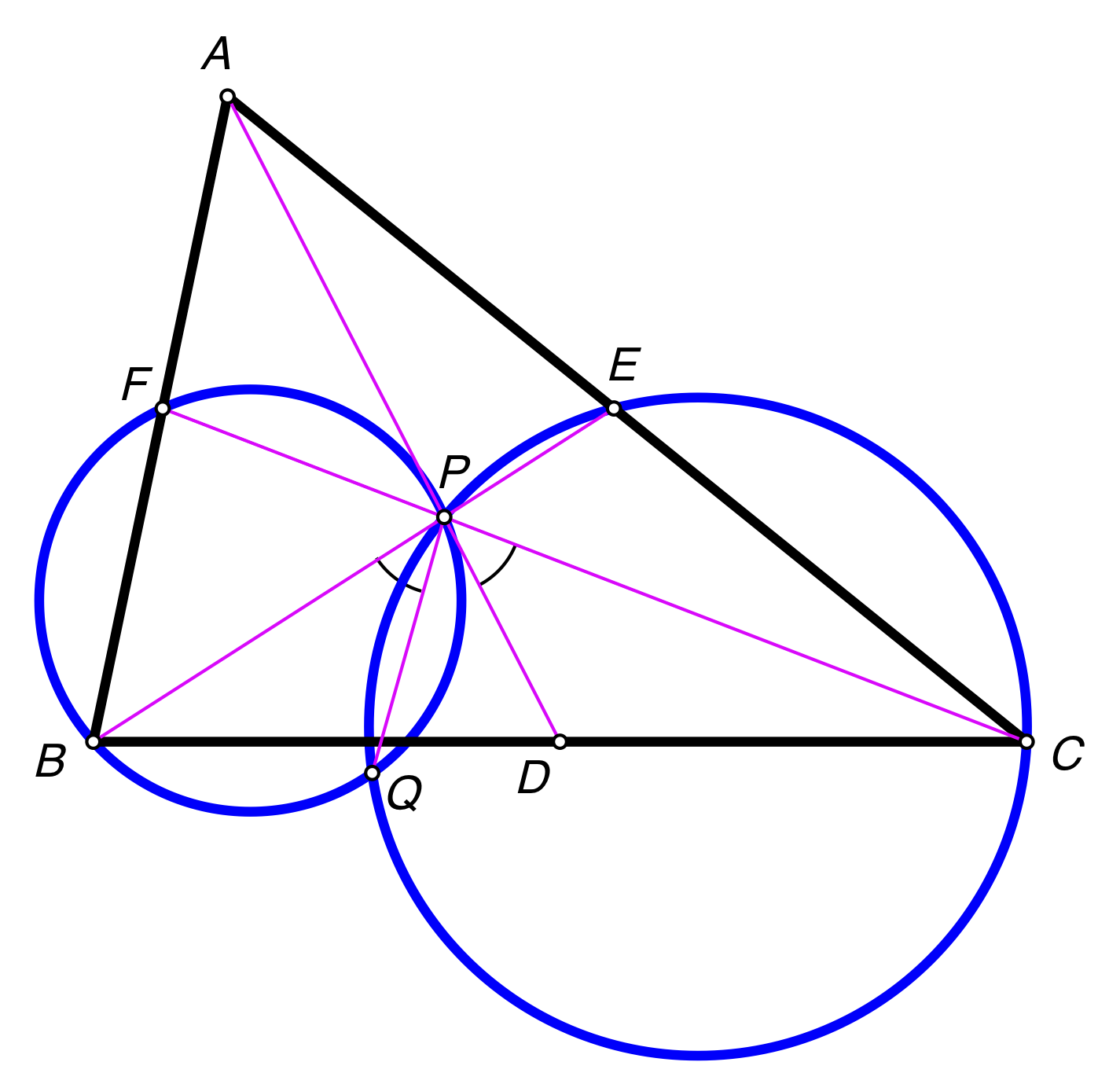}
\caption{}
\label{fig:lemmaFigure}
\end{figure}

\begin{theorem}
\label{thm:O1236}
Let P be any point on the median $AD$ of $\triangle ABC$.
The cevians through $P$ divide $\triangle ABC$ into six smaller triangles, labeled $T_1$ through $T_6$
as shown in Figure \ref{fig:sixTriangles}.
Let $O_i$ be the circumcenter of $T_i$.
Then (Figure \ref{fig:circles-median})\\
(a) the points $O_1$, $O_2$, $O_3$, and $O_6$ lie on a circle,\\
(b) the points $O_3$, $O_4$, $O_5$, and $O_6$ lie on a circle.
\end{theorem}

\begin{figure}[h!t]
\centering
\includegraphics[width=0.5\linewidth]{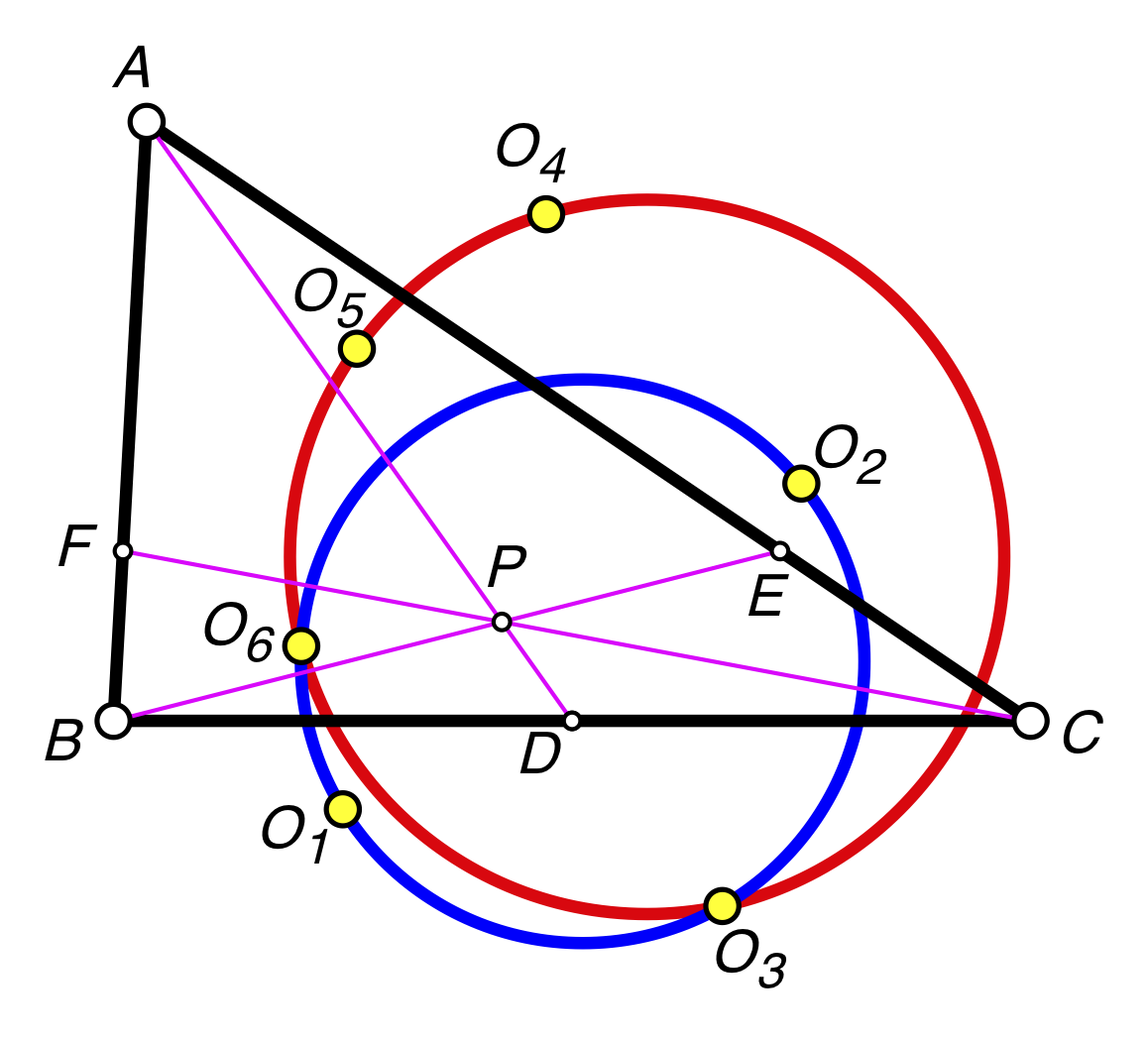}
\caption{The situation when $AD$ is a median}
\label{fig:circles-median}
\end{figure}

\begin{proof}
Part (a). Since $O_1$ is the circumcenter of $\triangle BPD$, $O_1$ must lie on the perpendicular bisector of $PB$. Similarly, $O_6$ lies on the perpendicular bisector of $BP$.
In the same way, $O_1O_2\perp PD$ and $O_2O_3\perp PC$.
Draw the circumcircles of triangles $PEC$ and $PFB$.
These circles meet at $P$ and $Q$ (Figure \ref{fig:median-proof}).
The line joining the centers of two intersecting circles is perpendicular to the common chord,
so $O_3O_6\perp PQ$.
Since $\angle PNO_2=\angle PKO_2$, quadrilateral $PNKO_2$ is cyclic
and hence $\angle NPK=\angle NO_2K$.
By the same reasoning, quadrilateral $PLMO_6$ is cyclic and so
$\angle LPM=\angle LO_6M$.
But by Lemma \ref{lemma:symmedian}, $\angle NPK=\angle LPM$.
Therefore $\angle NO_2K=\angle LO_6M$.
This makes quadrilateral $O_6O_1O_3O_2$ cyclic and
$O_1$, $O_2$, $O_3$, and $O_6$ lie on a circle as claimed.
\end{proof}

\begin{figure}[h!t]
\centering
\includegraphics[width=0.6\linewidth]{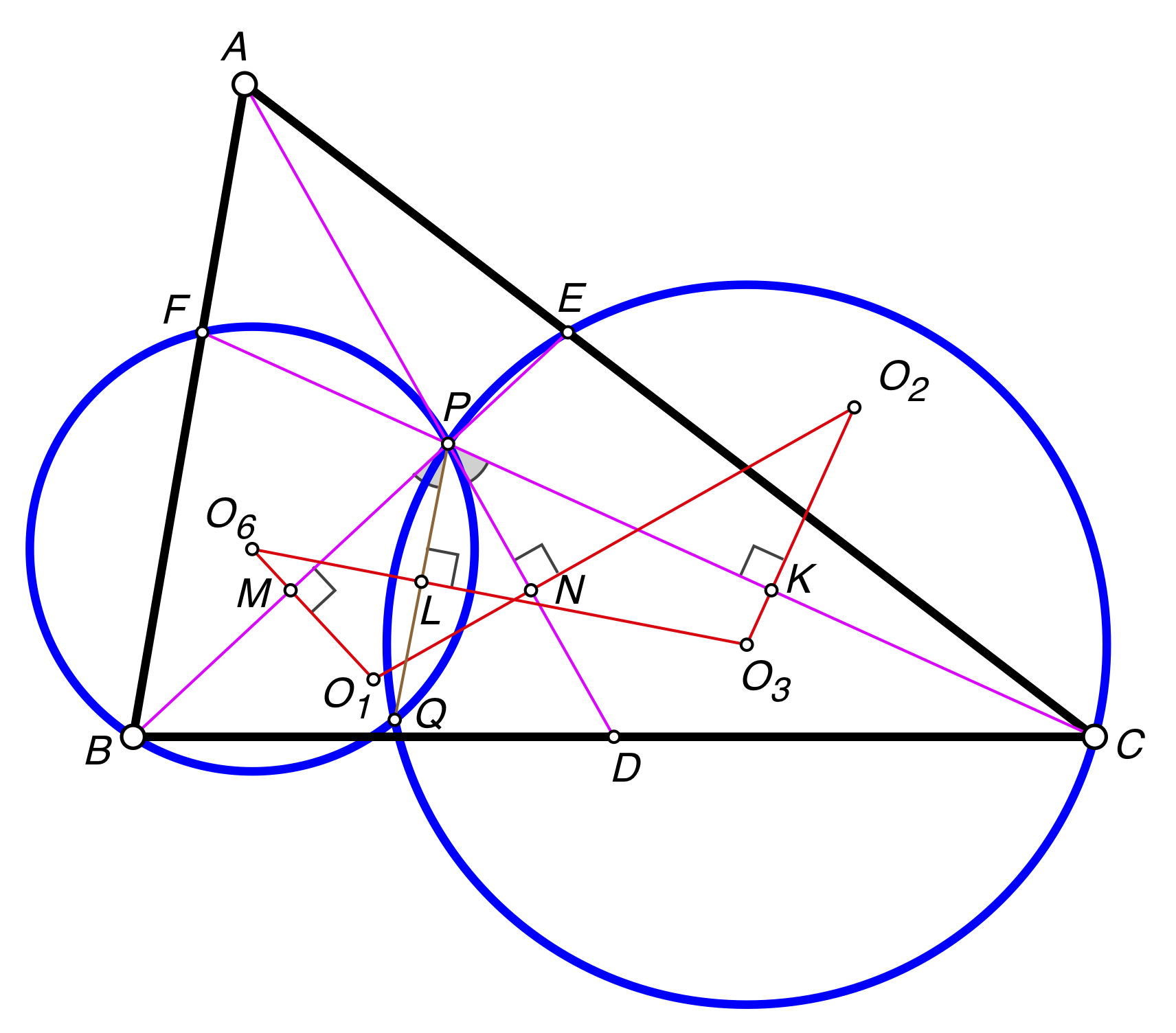}
\caption{}
\label{fig:median-proof}
\end{figure}

Part (b). This result is due to Suppa \cite{Suppa}. The proof is similar to the proof of part~(a) and details can be found in \cite{Pavlakis}.

%==========================
% Bibliography

\end{document}